\documentclass[reqno,10pt]{amsart}
\usepackage{amsfonts}
\usepackage[english]{babel}
\usepackage{amsthm,hyphenat}
\usepackage{amsmath}
\usepackage{amssymb}
\usepackage[latin1]{inputenc}
\usepackage[all]{xy}
\usepackage{dsfont}
\usepackage[bookmarksnumbered,colorlinks]{hyperref}
\usepackage{graphics}
\usepackage{enumerate}
\usepackage{mathrsfs} 
\usepackage[normalem]{ulem}
\def\bb#1\eb{\textcolor{blue}
{#1}} %
\def\br#1\er{\textcolor{red}
{#1}} %
\newcommand{\R}{\mathds R}

\title[Geodesics and Jacobi fields of pseudo-Finsler manifolds]{Geodesics and Jacobi fields of  pseudo-Finsler manifolds}

\author[M. A. Javaloyes]{Miguel Angel Javaloyes}
\address{Departamento de Matem\'aticas, \hfill\break\indent
Universidad de Murcia, \hfill\break\indent
Campus de Espinardo,\hfill\break\indent
30100 Espinardo, Murcia, Spain}
\email{majava@um.es}

\author[B. L. Soares]{Bruno Learth Soares}
\address{
Centro de Matem\'atica, Computa\c{c}\~ao e Cogni\c{c}\~ao, 
\hfill\break\indent
Universidade Federal do ABC (UFABC), \hfill\break\indent
Rua Santa Ad\'elia, 166 - 09210-170 - Santo Andr\'e (SP), Brazil}
\email{bruno.soares@ufabc.edu.br}

\date{}

\thanks{This work was partially supported by MINECO (Ministerio de Econom\'{\i}a y Competitividad) and FEDER (Fondo Europeo de Desarrollo Regional) project MTM2012-34037,   Regional J. Andaluc\'{\i}a Grant P09-FQM-4496 and Fundaci\'on S\'eneca project 18942/JLI/13. This  research has been supported by the programme ``Young leaders in research''  by Fundaci\'on S\'eneca, Regional Agency for Science and Technology from the Region of Murcia.  This work has been finsihed during a visit of the first author to the University of S\~ao Paulo supported by the Fapesp grant for visiting researchers with process 2013/1070-7. }


\begin{document}
\newtheorem{thm}{Theorem}[section]
\newtheorem{prop}[thm]{Proposition}
\newtheorem{lemma}[thm]{Lemma}
\newtheorem{cor}[thm]{Corollary}
\theoremstyle{definition}
\newtheorem{defi}[thm]{Definition}
\newtheorem{notation}[thm]{Notation}
\newtheorem{exe}[thm]{Example}
\newtheorem{conj}[thm]{Conjecture}
\newtheorem{prob}[thm]{Problem}
\newtheorem{rem}[thm]{Remark}
\begin{abstract}
In this paper, we derive the first and the second variation of the energy functional for a pseudo-Finsler metric using the family of affine connections associated to the Chern connection. This opens the possibility to accomplish  computations with coordinate-free methods. 
Using the second variation formula, we introduce the index form and 
 present  some properties of Jacobi fields. 
\end{abstract}

\maketitle
\section{Introduction}

Geodesics and Jacobi fields are probably the most important geometrical 
elements associated to a Finsler metric. Even though they can be defined 
without using any connection, choosing  an appropriate connection 
associated to the Finsler metric can make  it easier to get some properties of 
them. In particular, the main goal of this paper is to use the Chern connection 
to deduce some of these properties under the approach developed by H.-H. Matthias in \cite[Definition 2.5]{Mat80}, where the Chern connection is interpreted as a family of affine connections, namely, for every vector field $V$ in an open subset $\Omega\subset M$, non-zero everywhere, we get an affine connection $\nabla^V$. This affine connection is torsion-free and almost $g$-compatible, meaning that the derivative of the fundamental tensor is an expression in terms of the Cartan tensor (see subsection \ref{chernsection}). Both properties allow one to 
 compute  the first and second variation of the energy functional in  a coordinate-free  manner.  In this process, we will also use the further developments given in \cite{J13,J14}, where a satisfactory relation between the curvature of the affine connection and the Chern curvature is obtained. 

As one of our intentions is to promote the study of Finsler geometry between researchers of Riemannian background, we have 
 included many details, 
with the purpose of providing in some cases index-free proofs or establishing the results in the very general setting of pseudo-Finsler metrics, apparently, the most general case where the Chern connection can be defined 
 (see Remark \ref{quitevague}). In particular the square of a Finsler metric is a pseudo-Finsler metric and the notions of indefinite Finsler metrics \cite{Beem70,BeKi73}
and Finsler spacetimes \cite{Perlick06} can fit into this definition.

The fundamental objective of this paper is to provide a computation of the first and the second variation of the energy functional of a pseudo-Finsler metric (Propositions \ref{firstvar} and \ref{secondvar}). As a first step, we prove that geodesics are the critical points of the energy functional when we consider curves with fixed endpoints, or, more generally, with endpoints in two submanifolds $P$ and $Q$ (Corollary \ref{critgeo}). Moreover, the second variation formula allows us to define the index form and the Jacobi fields (subsection \ref{critgeo}), and with our approach to the Chern connection we can straightforwardly deduce  some basic properties of Jacobi fields (see subsection \ref{jacobea}) and characterize the
kernel of the index form as the  $(P,Q)$-Jacobi fields along $\gamma$ (Proposition \ref{kernelindex}). 

The paper is structured as follows. Section \ref{quadratic} contains some basic results about pseudo-Finsler metrics including some properties of its fundamental tensor, the Cartan tensor and the Chern connection.
We also introduce several basic notions:  parallelism of a vector field along a curve, geodesics, namely, curves having parallel tangent vector fields, and the exponential map. In the last subsection we recall some properties 
 concerning  the curvature of the Chern connection obtained in \cite{J13,J14},  in particular its relation to the Chern curvature.

In Section \ref{variations} we compute the first and the second variation of the energy functional (Propositions \ref{firstvar} and \ref{secondvar}) and then we get the index form when the boundary conditions are given by two submanifolds. As a preparatory  step, in subsection \ref{submani} we  collect some definitions and properties of submanifolds of pseudo-Finsler manifolds. Then,  in subsection \ref{jacobea}, we  give some properties of Jacobi fields.

\section{Pseudo-Finsler metrics and the Chern connection}\label{quadratic}
\subsection{Preliminaries on pseudo-Finsler metrics} 
Let $M$ be an  $n$-dimensional manifold and denote by $\pi:TM\rightarrow M$ the natural projection
of the tangent bundle $TM$ into $M$. Let $A\subset TM\setminus  {\bf 0}$ be an open subset of $TM$ which is conic, that is, such that $\pi(A)=M$ and $\lambda v\in A$, for every $v\in A$ and $\lambda>0$. We say that a smooth function
$L:A\rightarrow \R$ is a (conic, two homogeneous) {\it pseudo-Finsler metric} if
\begin{enumerate}[(i)]
\item $L$ is positive-homogeneous of degree $2$, that is, $L(\lambda v)=\lambda^2 L(v)$ for every $v\in A$ and $\lambda>0$,
\item for every $v\in A$, {\it the fundamental tensor $g_v$ of $L$ at 
$v$} defined by
\[g_v(u,w):=\frac 12 \frac{\partial^2}{\partial t\partial s} L(v+tu+sw)|_{t=s=0},\]
for any $u,w\in T_{\pi(v)}M$, is nondegenerate.
\end{enumerate}
 Clearly,  the fundamental tensor is bilinear and symmetric. We will  refer to the pair $(M,L)$, being $M$ a manifold and $L$ a pseudo-Finsler metric on $M$, as a pseudo-Finsler manifold. 
\begin{rem}\label{pseudo} 
In the following we shall omit the adjectives   two-homogeneous and conic  whenever there is no danger of misunderstanding. In \cite{JS13} the same name of pseudo-Finsler metrics is used for a somewhat different concept. In the cited reference,  a pseudo-Finsler metric is not allowed to be non-positive away from the zero section and it is  positive homogeneous of degree one. Moreover, its fundamental tensor is not necessarily nondegenerate. Nevertheless, if $F:A_1\subset TM\rightarrow [0,+\infty)$ is a conic pseudo-Finsler metric on $M$ as in \cite{JS13}  and 
\[\tilde{A}_1=\{v\in A_1\setminus {\bf 0}: \text{the fundamental tensor $g_v$ of $F^2$ is nondegenerate}\},\]
then $L=F^2|_{\tilde{A}_1}:\tilde{A}_1\rightarrow (0,+\infty)$ fits in our definition of pseudo-Finsler metric. Moreover, if $L:A_2\subset TM\setminus {\bf 0}\rightarrow\R$ is a pseudo-Finsler metric on $M$ as defined above and
$\tilde{A}_2=\{v\in A_2: L(v)\not=0\},$ then $F=\left.\sqrt{|L|}\right|_{\tilde{A}_2}:\tilde{A}_2\rightarrow (0,+\infty)$ is a conic pseudo-Finsler metric on $M$ as in \cite{JS13} (extending $F$  continuously to the zero section if necessary). Note finally that the  concept  of pseudo-Finsler metrics in \cite{JS13} is particularly convenient when one is interested in studying distance properties. 
\end{rem}
 We mention some particular cases of pseudo-Finsler metrics:
\begin{enumerate}[(i)]
\item if $A=TM\setminus {\bf 0}$ and the fundamental tensor is positive definite, then $L$ is positive away from the zero section and $F=\sqrt{L}$ is what traditionally has been called a Finsler metric;
\item if $A\subsetneq TM\setminus {\bf 0}$, but the fundamental tensor is positive definite, then the square root
$F=\sqrt{L}$ is called a {\it conic Finsler} metric in \cite{JS13};
\item if the fundamental tensor has index one, then $L$ is called a 
{\it Lorentzian Finsler  metric} (see \cite{AJ14,CJS13,JV13}). This is also the case 
of Finsler spacetimes where some authors ask $L$ to be defined in the whole 
$TM$ \cite{Beem70,BeKi73,Perlick06}.
\end{enumerate}

\begin{rem}
Even though sometimes the domain of definition can change (see Remark \ref{pseudo}), from now on, by abuse of notation, we omit the subset $A$ when fixing a 
(conic) pseudo-Finsler manifold $(M,L)$, assuming that $L$ is defined in $A$. 
\end{rem}
  We end this subsection with a proposition that describes the 
positive homogeneity of the fundamental tensor of a pseudo-Finsler metric 
as well as some relations between the fundamental tensor and the metric. 
It is a simple consequence of basic properties of homogeneous functions. 
\begin{prop}\label{fundamentalprop}
Given a pseudo-Finsler metric $L$ and $v\in A$, the fundamental tensor 
$g_v$ is positive homogeneous of degree 0, that is, $g_{\lambda v}=g_v$ for $\lambda>0$. Moreover $g_v(v,v)=L(v)$ and
$ g_v(v,w)= \frac 12\frac{\partial}{\partial z}L\left(v+z w\right)|_{z=0}$.
\end{prop}


\subsection{Cartan tensor} In Finsler geometry, unlike the Riemannian setting, we need to consider the third derivatives of the metric in order to define a connection. This information is contained in the {\it Cartan tensor}, which is defined as the trilinear form
\begin{equation}\label{cartantensor}
C_v(w_1,w_2,w_3)=\frac 14\left. \frac{\partial^3}{\partial s_3\partial s_2\partial s_1}L\left(v+\sum_{i=1}^3 s_iw_i\right)\right|_{s_1=s_2=s_3=0},
\end{equation}
for $v\in A$ and $w_1,w_2,w_3\in T_{\pi(v)} M$.
Observe that $C_v$ is symmetric, that is, its value does not depend on the order of $w_1$, $w_2$ and $w_3$. 
\begin{rem}\label{fiberbundle}
Let $\pi_A:A\rightarrow M$ be the restriction to $A$ of the natural 
projection $\pi:TM\rightarrow M$. Now let $\pi^*_A(T^*M)$ be the fiber bundle over 
$A$ induced by the natural projection of the cotangent bundle $\pi^*:T^*M
\rightarrow M$ through $\pi_A$.
Observe that the fundamental tensor is a symmetric section of the fiber bundle $\pi^*_A(T^*M)\otimes
\pi^*_A(T^*M)$. 
Moreover, the Cartan tensor is a symmetric section of the fiber bundle $\pi^*_A(T^*M)\otimes\pi^*_A(T^*M)\otimes \pi^*_A(T^*M)$.
\end{rem}
 Furthermore, the Cartan tensor can be obtained from the fundamental tensor as
\[ C_v(w_1,w_2,w_3)=\frac 12\left.\frac{\partial}{\partial z} g_{v+z w_1}(w_2,w_3)\right|_{z=0}.\]
If $g$ is an arbitrary symmetric section of $\pi^*_A(T^*M)\otimes
\pi^*_A(T^*M)$ that is positive homogeneous of degree zero  ($g_v=g_{\lambda v}$ for $\lambda>0$)  we define its Cartan tensor as above.

 The following is another simple consequence of basic properties of 
homogeneous functions: 
\begin{prop}\label{Cartanprop}
 The Cartan tensor is homogeneous of degree $-1$, that is, $C_{\lambda v}=\frac{1}{\lambda}C_v$ for any $v\in A$ and $\lambda>0$. Moreover,  $C_v(v,w_1,w_2)=0$ for every $w_1,w_2\in T_{\pi(v)}M$.
\end{prop}


 To conclude this subsection, we state a 
well-known result that appears, for example, in \cite[Theorem 3.4.2.1]{AIM93}.

\begin{prop}\label{generalizedmetrics}
An arbitrary (non-degenerate) symmetric section $g$ of \linebreak
$\pi^*_A(T^*M)\otimes \pi^*_A(T^*M)$ that is positive homogeneous of degree zero comes from a pseudo-Finsler metric
if and only if its Cartan tensor is  symmetric.
\end{prop}

\begin{rem}\label{quitevague}
 An arbitrary symmetric section $g$ of $\pi^*_A(T^*M)\otimes
\pi^*_A(T^*M)$ that is positive homogeneous of degree zero is usually known as a generalized metric. It was introduced by A. Moor in 1956 \cite{Moor56} and studied in detail by J. R. Vanstone \cite{Vanstone}, M. Hashiguchi \cite{Hashiguchi}, R. Miron \cite{Miron83} and others. For a quite recent survey, we refer to \cite{LPS}. 
Unfortunately, the Chern connection is not well-defined for generalized metrics unless they come from a pseudo-Finsler metric. This is because the following remark is essential  to prove existence of a connection which is torsion-free and almost metric compatible (see for example \cite[Proposition 2.3]{J13}).
\end{rem}
\begin{rem}\label{porvenir}
In the case of a pseudo-Finsler manifold $(M,L)$, the Cartan tensor is symmetric. This together with Proposition \ref{Cartanprop} means that $C_v(v,w_1,w_2)=C_v(w_1,v,w_2)=C_v(w_1,w_2,v)=0$ for any $v\in A$ and $w_1,w_2\in T_{\pi(v)}M$.
\end{rem}
\subsection{Chern connection and covariant derivative}\label{chernsection}

 Assume that $(M,L)$ is a pseudo-Finsler manifold with domain $A\subset TM$, and denote by $\mathfrak X(\Omega)$ the module of smooth vector fields on an open subset $\Omega\subset M$. We say that $V\in\mathfrak X(\Omega)$ is {\it $L$-admissible} if $V(p)\in A$ for every $p\in\Omega$. Then the mapping
\[
 \begin{cases}
  g_V: (X,Y)\in\mathfrak X(\Omega)\times \mathfrak X(\Omega)\mapsto g_V(X,Y)\in C^\infty(\Omega),\\
  g_V(X,Y)(p):= g_{V(p)}(X(p),Y(p))\quad (p\in\Omega)
 \end{cases}
\]
is a pseudo-Riemannian metric on $\Omega$. We construct in the same way a type $(0,3)$ tensor field $C_V$ on $\Omega$ from the Cartan tensor.

 It can be associated to any $L$-admissible vector field $V\in\mathfrak X(\Omega)$ an affine connection
\[
 \nabla^V:\mathfrak X(\Omega)\times \mathfrak X(\Omega)\to \mathfrak X(\Omega)
\]
such that
\begin{itemize}
 \item[(i)] $\nabla^V_XY-\nabla^V_YX=[X,Y]$ for all $X,Y\in\mathfrak X(\Omega)$ 
({\it torsion freeness});
 \item[(ii)] $X(g_V(Y,Z))=g_V(\nabla^V_XY,Z)+g_V(Y,\nabla^V_XZ)+2C_V(\nabla^V_XV,Y,Z)$ for all $X,Y,Z\in \mathfrak X(\Omega)$ ({\it almost $g$-compatibility}).
\end{itemize}
We say that $\nabla^V$ is the {\it Chern connection} of  $(M,L)$  associated to the $L$-admissible vector field $V\in\mathfrak X(\Omega)$. It is easy to see that $\nabla^V$ is positive homogeneous of degree $0$ in $V$, i.e., $\nabla^V=\nabla^{\lambda V}$ for all positive $\lambda$. 

 Now we suppose that $\Omega$ is a chart domain with coordinate system 
\[
 x=(x^1,\dots,x^n):\Omega\to x(\Omega)\subset \R^n.
\]
The Christoffel symbols of $\nabla^V$ with respect to the chart $(\Omega,x)$ are the smooth functions $\Gamma^k_{ij}(V):\Omega\to\R$ such that
\[
 \nabla^V_{\frac\partial{\partial x^i}}\left(\frac\partial{\partial x^j}\right)=\sum\Gamma^k_{ij}(V)\frac\partial{\partial x^k};\quad i,j\in\{1,\dots,n\}.
\]
In fact, $\Gamma^k_{ij}(V)=\Gamma^k_{ij}\circ V$ where $\Gamma^k_{ij}:\pi^{-1}(\Omega)\cap A\to\R$ are the Christoffel symbols of the Chern connection (see \cite[Proposition 2.6]{J13}).

 Given a smooth curve $\gamma:[a,b]\to M$, we denote by $\mathfrak X(\gamma)$ the $C^\infty([a,b])$-module of vector fields along $\gamma$. We say that $W\in\mathfrak X(\gamma)$ is $L$-admissible if $W(t)\in A$ for all $t\in [a,b]$. For every $L$-admissible vector field $W\in\mathfrak X(\gamma)$, the Chern connection induces a covariant derivative $D^W_\gamma:\mathfrak X(\gamma)\to\mathfrak X(\gamma)$ along $\gamma$, given locally,  when $\gamma$ is contained in the chart domain $\Omega$,  by
\begin{equation}\label{connection}
 D^W_\gamma X=\sum_{i=1}^n\left(\dot X^k+\sum_{i,j=1}^nX^i \dot\gamma^j (\Gamma^k_{ij}\circ W)\right)\left(\frac\partial{\partial x^k}\circ\gamma\right),
\end{equation}
where $X=\sum_{i=1}^n X^i(\frac\partial{\partial x^i}\circ\gamma)$, $\dot\gamma=\sum_{i=1}^n \dot\gamma^i(\frac\partial{\partial x^i}\circ\gamma)$ (see, again, \cite[Proposition 2.6]{J13}). The induced covariant derivative is also almost $g$-compatible in the sense that
\begin{equation}\label{almostmetric}
 (g_W(X,Y))'=g_W(D^W_\gamma X,Y)+g_W(X,D^W_\gamma Y)+2C_W(D^W_\gamma W,X,Y),
\end{equation}
for all $X,Y\in\mathfrak X(\gamma)$.

\subsection{Parallel vector fields}\label{parallelism}

 Having been defined the induced covariant derivative, we can introduce the concept of parallelism along a (piecewise) smooth curve $\gamma:[a,b]\to M$. Then, of course, we have to assume that $\gamma$ is also $L$-admissible in the sense that $\dot\gamma(t)\in A$ for all $t\in[a,b]$. (At break points, both velocity vectors of $\gamma$ must belong to $A$.) For simplicity, we may assume that $\operatorname{Im}(\gamma)$ is in the domain of a chart $(\Omega,x)$; this assumption is clearly not restrictive.

\begin{defi}
 Let $(M,L)$ be a pseudo-Finsler manifold and $\gamma$ a smooth $L$-admissible curve in $M$. A vector field $X\in\mathfrak X(\gamma)$ is called {\it parallel} if $D^{\dot\gamma}_\gamma X=0$.
\end{defi}
\begin{prop}
 Hypotheses and notation as above. Given a vector $w\in  T_{\dot\gamma(a)}M$, there is a unique parallel vector field $X$ along $\gamma$ such that $X(a)= w$.
\end{prop}

The proof is routine.

\subsection{Geodesics and the exponential map} 
\begin{defi}
 A smooth $L$-admissible curve $\gamma$ of a pseudo-Finsler manifold $(M,L)$ is called a {\it geodesic} if its velocity vector field $\dot\gamma$ is parallel along $\gamma$. 
\end{defi} 

In terms of local coordinates, the geodesic equation is of the form
\begin{equation}
 \ddot\gamma^k+\sum_{i,j=1}^n\dot\gamma^i \dot\gamma^j(\Gamma^k_{ij}\circ\dot\gamma)=0,\quad k\in\{1,\dots,n\}.
\end{equation}
As in the pseudo-Riemannian case, we have:
\begin{prop}
 Let $(M,L)$ be a pseudo-Finsler manifold. For every $v\in A$, there exists a unique (maximal) geodesic $\gamma_v:\left[a,b\right)\to M$ such that $\dot\gamma_v(0)=v$, $b\in\left(0,+\infty\right]$.
\end{prop}
\begin{rem}
 If $\gamma:[a,b]\to M$ is a geodesic of $(M,L)$, then the function 
$L\circ\dot\gamma:[a,b]\to\R$ is constant. Indeed, taking into account the almost $g$-compatibility of the induced covariant derivative and Remark~2.9, we have \linebreak$(L\circ\dot\gamma)'=2g_{\dot\gamma}(D^{\dot\gamma}_\gamma\dot\gamma,\dot\gamma)=0$. If $L\circ\dot\gamma=0$, then $\gamma$ is called a {\it lightlike geodesic}.
\end{rem}
\begin{defi}
 Let $(M,L)$ be a pseudo-Finsler manifold. If $p\in M$, let $\mathcal D_p$ be the set of vectors $v$ in $A\cap T_pM$ satisfying the following condition: if $\gamma_v:\left[0,b\right)\to M$ is the maximal geodesic such that $\dot\gamma_v(0)=v$, then $b>1$. The {\it exponential map} of $(M,L)$ at $p$ is the mapping
 \[
  \exp_p:\mathcal D_p\to M,\quad v\mapsto\exp_p(v):=\gamma_v(1).
 \]
\end{defi}
 \begin{prop}
 The domain of the exponential map $\exp_p:\mathcal D_p\to M$ is an open subset of $A\mathop\cap T_pM$, and $\exp_p$ is smooth on $\mathcal D_p$. If $A\cap T_pM=T_pM\setminus\{0\}$, then $\exp_p$ is defined in an open subset of $0_p$, putting $\exp_p(0_p)=p$.
\end{prop}

\begin{proof}
  The first statement is a consequence of the theorem on smooth dependence on the initial data of ODEs. To see the second statement, observe that the functions
 \[
  v\in T_pM\setminus\{0\}\mapsto \sum v^iv^j\Gamma^k_{ij}(v)\in\R;\quad k\in\{1,\dots,n\}
 \]
 are positive-homogeneous of degree 2, so they can be extended to $C^1$ functions at zero. Indeed, homogeneous functions of positive degree can be extended continuously to zero as zero, and the derivative of a homogeneous function of degree 2 is a homogeneous function of degree 1 (see also Proposition~\ref{fundamentalprop}). In fact, the extensions will be of class $C^1$ on $\pi^{-1}(\Omega)$. 
\end{proof}

\subsection{Jacobi operator and flag curvature}\label{curvature}
If we fix an $L$-admissible vector field $V$ in $\Omega\subset M$, the Chern connection $\nabla^V$ is an affine connection on $\Omega$, whose curvature tensor is given by
\begin{equation*}
R^V(X,Y)Z=\nabla^V_X\nabla^V_YZ-\nabla^V_Y\nabla^V_XZ-
\nabla^V_{[X,Y]}Z \quad ; \quad X,Y,Z\in\mathfrak{X}(\Omega).
\end{equation*}
The $\nabla^V$-covariant derivative of the Cartan tensor $C_V$ is the 
$(0,4)$ tensor defined by
\begin{multline*}
\nabla^V_XC_V(Y,Z,W)=X(C_V(Y,Z,W))-C_V(\nabla^V_XY,Z,W)\\-C_V(Y,\nabla^V_XZ,W)-
 C_V(Y, Z ,\nabla^V_XW)~.
\end{multline*}
It is straightforward to check that $\nabla^V_XC_V$ 
is trilinear, symmetric and 
\begin{equation}\label{formulilla}
\nabla^V_XC_V(V,Z,W)=-C_V(\nabla^V_XV,Z,W).
\end{equation}
Moreover, for every $X,Y,Z,W\in\mathfrak{X}(\Omega)$, the curvature tensor has 
the following symmetries (see \cite[Proposition 3.1]{J13}):
\begin{enumerate}[(i)]
\item  \[R^V(X,Y)=-R^V(Y,X)~;\]
\item \[g_V(R^V(X,Y)Z,W)+g_V(R^V(X,Y)W,Z)=2B^V(X,Y,Z,W)~,\]
where
\begin{multline*}
B^V(X,Y,Z,W)=\\
\nabla^V_YC_V(\nabla_X^VV,Z,W)-\nabla^V_XC_V(\nabla_Y^VV,Z,W)+C_V(R^V(Y,X)V,Z,W)~;
\end{multline*}
\item \[R^V(X,Y)Z+R^V(Y,Z)X+R^V(Z,X)Y=0~;\]
\item \begin{multline}\label{seisB}
g_V(R^V(X,Y)Z,W)-g_V(R^V(Z,W)X,Y)=\\
 B^V(Z,Y,X,W)+B^V(X,Z,Y,W)+B^V(W,X,Z,Y)\\
 +B^V(Y,W,Z,X)+B^V(W,Z,X,Y)+B^V(X,Y,Z,W)~.
\end{multline}
\end{enumerate}


 We can also define the Jacobi operator $R^\gamma$ along an $L$-admissible curve 
$\gamma:[a,b]\subset \R\rightarrow M$. Recall that the curve $\gamma$ is 
$L$-admissible if $\dot\gamma$ belongs to $A$ and consider
a  smooth variation $\Lambda:[a,b]\times (-\varepsilon,\varepsilon)\rightarrow M$, which is a two-parameter map. 
  Given $s_0\in (-\varepsilon,\varepsilon)$ and $t_0\in [a,b]$, we will denote by $\gamma_{s_0}:[a,b]\rightarrow M$ the curve defined as $\gamma_{s_0}(t)=\Lambda(t,s_0)$ for every $t\in [a,b]$ and by
$\beta_{t_0}:(-\varepsilon,\varepsilon)\rightarrow M$ the curve defined as $\beta_{t_0}(s)=\Lambda(t_0,s)$ for every $s\in (-\varepsilon,\varepsilon)$,  which are the longitudinal and the transverse curves of the variation,  respectively. 
Moreover,  we will use the notation $\Lambda_t(t,s)=\dot\gamma_s(t)$ and 
$\Lambda_s(t,s)=\dot\beta_t(s)$ and we will denote by $\Lambda^*(TM)$ the vector bundle over $[a,b]\times (-\varepsilon,\varepsilon)$ induced by $\pi:TM\rightarrow M$ through $\Lambda$. Then the space of smooth sections 
of $\Lambda^*(TM)$ will be denoted as ${\mathfrak X}(\Lambda)$. Observe that a vector field $V\in {\mathfrak X}(\Lambda)$ induces vector fields in $\mathfrak X(\gamma_{s_0})$ and $\mathfrak X(\beta_{t_0})$ for every $s_0\in (-\varepsilon,\varepsilon)$ and $t_0\in [a,b]$.
We will say that $V$ is $L$-admissible if $V(t,s)\in A$ for every $(t,s)\in [a,b]\times (-\varepsilon,\varepsilon)$.
  When $\Lambda$ lies in the domain of a coordinate system $x^1,\ldots,x^n$, we will denote $\Lambda^i=x^i\circ\Lambda$.
Notice that when we have a variation of curves (or more generally a two parameters map), as the Chern connection is free of torsion, we have the following property:
\begin{equation}\label{commut}
D_{\gamma_s}^V{\dot\beta_t}=D_{\beta_t}^V{\dot\gamma_s},
\end{equation} 
(see also \cite[Proposition 3.2]{J13}).
  We say that the variation $\Lambda$ is $L$-admissible if $\gamma_s$ is $L$-admissible for every $s\in (-\varepsilon,\varepsilon)$. Moreover, we will denote by $W$ the variational vector field of $\Lambda$ along $\gamma$, namely, $W(t)=\Lambda_s(t,0)$ for every $t\in[a,b]$. 
   If $\Lambda$ is $L$-admissible  and $Z\in \mathfrak{X}(\Lambda)$, we can define 
  \[R^\Lambda(Z):=D_{\gamma_{s}}^{\Lambda_t}D_{\beta_{t}}^{\Lambda_t}Z-D_{\beta_{t}}^{\Lambda_t}D_{\gamma_{s}}^{\Lambda_t}Z,\]
which is a smooth vector field along  $\Lambda$.  Now observe that given an $L$-admissible curve $\gamma:[a,b]\rightarrow M$, and an arbitrary smooth vector field $W$ along $\gamma$ there always exists a (non-unique) $L$-admissible variation $\Lambda$ of $\gamma$ with $W$ as variational vector field. In fact, it is well-known that we can choose a 
variation $\Lambda:[a,b]\times (-\varepsilon,\varepsilon)\rightarrow M$ of $\gamma$ having $W$ as a variation vector field. As $\Lambda$ is at least $C^1$, being $A$ an open subset  and  $[a,b]$ compact, we can choose a smaller $\varepsilon$ if necessary in such a way that $\Lambda$ is $L$-admissible. Moreover, following \cite{J14}, we can define 
\begin{equation}
\label{jacobiop}
R^\gamma(\dot\gamma,W)Z:=R^\Lambda(\tilde{Z}),
\end{equation}
where $Z$ is a smooth vector field along $\gamma$ and $\tilde{Z}$ is a smooth extension of $Z$ to $\Lambda$. As was proven in \cite{J14}, the operator $R^\gamma$ is well-defined because it does not depend on the choice of the variation $\Lambda$ neither on the extension of $Z$.

 In general,  $R^\gamma(\dot\gamma,W)Z$ is not tensorial in $W$, since 
it is not $C^\infty([a,b])$-linear in $W$,  but as a consequence of \cite[Corollary 1.3]{J14}, it is when $Z=\dot\gamma$. 
Moreover, when $\gamma$ is a geodesic
\[R^\gamma(\dot\gamma,W)\dot\gamma=R_{\dot\gamma}(\dot\gamma,W)\dot\gamma,\]
where we denote by $R_v$ the Chern curvature (see \cite{J14} or 
\cite[Formula (3.3.2) and Exercise 3.9.6]{BaChSh00}). Finally, for any $v\in A$ and $w\in T_{\pi(v)}M$, the flag curvature can be computed as 
\[K_v(w)=\frac{g_v(R^{\gamma_v}(\dot\gamma_v,W)W(t_0),v)}{L(v)g_v(w,w)-g_v(v,w)^2},\]
where $\gamma_v$ is the geodesic such that $\dot\gamma_v(t_0)=v$ and $W$ is a smooth vector field
along $\gamma$ such that $W(t_0)=w$ (see \cite[Remark 2.3]{J14}). 
\section{Variation of the energy}\label{variations}

Given a pseudo-Finsler manifold $(M,L)$, we shall denote by $C_L(M,[a,b])$ the space of $L$-admissible piecewise smooth curves in $M$ defined on the closed interval $[a,b]$. We write $T_\gamma C_L(M,[a,b])$ for the set of all $L$-admissible, piecewise smooth continuous vector fields along $\gamma\in C_L(M,[a,b])$ with the same breaks as $\gamma$. The {\it energy functional} on $C_L(M,[a,b])$ is the function
\begin{equation}
 E:\gamma\in C_L(M,[a,b])\mapsto E(\gamma):=\frac12\int_b^a L\circ\dot\gamma~dt\ \in\R.
\end{equation}
We are going to show that the geodesics of $(M,L)$ are the critical points of $E$. To do this, we calculate the first and second variation formula for $E$.

 Let $\gamma\in C_L(M,[a,b])$, and consider the piecewise smooth variation $\Lambda:[a,b]\times(-\varepsilon,\varepsilon)\to M,\ (t,s)\mapsto\Lambda(t,s)$ of $\gamma$ with breaks $t_0:=a<t_1<t_2<\dots< t_{h}<t_{h+1}:=b$ and recall the notation for variations of the last section. 
Let us recall the definition of the {\it Legendre transformation} of a pseudo-Finsler metric, namely, the map $\mathscr{L}_L: A\rightarrow TM^*$, where $\mathscr{L}_L(v)$ is defined as the one-form given by $\mathscr{L}_L(v)(w)=g_v(v,w)$ for every $w\in T_{\pi(v)}M$ and $g$ is the fundamental tensor of $L$.
\begin{prop}\label{firstvar}
 Keeping the notation introduced above, let $\Lambda$ be an \linebreak $L$-admissible 
piecewise smooth variation of $\gamma$. Then we have the 
{\it first variation formula}
 \begin{equation}\label{firstvariation}
  \begin{aligned}[m]
  E'(0)&:=\frac{d(E(\gamma_s))}{ds}|_{s=0}\\
  &=-\int_a^bg_{\dot\gamma}(W,D^{\dot\gamma}_\gamma\dot\gamma)~dt +
g_{\dot\gamma}(W,\dot\gamma)|^b_a\\
  &\quad\quad\quad +\sum_{i=1}^h\big(\mathcal L_L(\dot\gamma(t_i^+))(W(t_i))-\mathcal L_L(\dot\gamma(t_i^-))(W(t_i))\big),
  \end{aligned}
 \end{equation}
 where $\dot\gamma(t_i^+)$ (resp. $\dot\gamma(t_i^-)$) denote the right (resp.\ left)  velocity at the breaks.
\end{prop}

\begin{proof}
As the variation is piecewise smooth, we get
\begin{align}
\nonumber\frac{d}{ds}E(\gamma_s)=&\frac 12\int_a^b \frac{d}{ds} g_{\dot\gamma_s}(\dot\gamma_s,\dot\gamma_s)~dt\\\nonumber
=&\int_a^b \big( g_{\dot\gamma_s}(D_{\beta_t}^{\dot\gamma_s}\dot\gamma_s,\dot\gamma_s)
+C_{\dot\gamma_s}(D_{\beta_t}^{\dot\gamma_s}\dot\gamma_s,\dot\gamma_s,\dot\gamma_s)\big)~dt\\\label{basic}
=& \int_a^bg_{\dot\gamma_s}(D_{\gamma_s}^{\dot\gamma_s} \dot\beta_t,\dot\gamma_s)~dt,
\end{align}
where we have used first that the Chern connection is almost $g$-compatible and then Remark \ref{porvenir} and \eqref{commut}. Moreover, applying again that the Chern connection is almost $g$-compatible,
we find that
\begin{equation}\label{inter}
g_{\dot\gamma}(D_\gamma^{\dot\gamma} W,\dot\gamma)=\frac{d}{dt}\big(g_{\dot\gamma}( W,\dot\gamma)\big)-g_{\dot\gamma} (W,D_\gamma^{\dot\gamma}\dot\gamma),
\end{equation}
because $C_{\dot\gamma}(D_{\gamma}^{\dot\gamma}\dot\gamma,W,\dot\gamma)=0$ (again by Remark \ref{porvenir}). Substituting \eqref{inter}  in \eqref{basic} with $s=0$ and integrating, we get finally \eqref{firstvariation}.
\end{proof}
Proposition \ref{firstvar} allows us to define formally the differential of $E$ in $\gamma$ as the map $dE_\gamma:T_\gamma C(M,[a,b])\rightarrow \R$  given  by
\begin{multline*}
dE_\gamma(W)=-\int_a^bg_{\dot\gamma}(W,D_\gamma^{\dot\gamma}\dot\gamma)~dt + \left[g_{\dot\gamma}(W,\dot\gamma)\right]_a^b
 \\ +\sum_{i=1}^h\big( \mathscr{L}_L(\dot\gamma(t_i^+))(W(t_i))-
\mathscr{L}_L(\dot\gamma(t_i^-))(W(t_i))\big),
\end{multline*}
for any $W\in T_\gamma C_L(M,[a,b])$. 

From now on, given a smooth $L$-admissible curve $\gamma:[a,b]\rightarrow M$ 
and \linebreak$W\in {\mathfrak X}(\gamma)$, we write $W'=D^{\dot\gamma}_\gamma W$.  As we shall see later, the critical points of the energy functional are geodesics when some boundary conditions are imposed.  Now we compute the second variation formula. 
\begin{prop}\label{secondvar}
Let $\gamma:[a,b]\rightarrow M$ be a  geodesic of $(M,L)$ and consider
an $L$-admissible smooth variation $\Lambda$. Then with the above notation
\begin{multline}\label{secondvariation}
E''(0)=\frac{d^2}{ds^2}E(\gamma_s)\left|_{s=0}\right.\\
=\int_a^b \left(-g_{\dot\gamma}(R^\gamma(\dot\gamma,W)W,\dot\gamma)+g_{\dot\gamma}(W',W')\right)dt+
\left[g_{\dot\gamma}(D^{\dot\gamma}_{\beta_t}\dot\beta_t|_{s=0},\dot\gamma)\right]_a^b,
\end{multline}
where $D^{\dot\gamma}_{\beta_t}\dot\beta_t|_{s=0}$ is the transverse acceleration vector field  (cf. \cite[page 266]{oneill}) of the variation.
\end{prop}
\begin{proof}
We will use Remark \ref{porvenir} along the proof without further comment. Using \eqref{basic} and the almost $g$-compatibility of the Chern connection, we get
\begin{multline*}
\frac{d^2}{ds^2}E(\gamma_s)=\frac{d}{ds}\int_a^bg_{\dot\gamma_s}(D_{\gamma_s}^{\dot\gamma_s} \dot\beta_t
,\dot\gamma_s)dt=\int_a^b\frac{d}{ds}g_{\dot\gamma_s}(D_{\gamma_s}^{\dot\gamma_s} \dot\beta_t
,\dot\gamma_s)dt\\
=\int_a^b\left(g_{\dot\gamma_s}(D_{\beta_t}^{\dot\gamma_s}D_{\gamma_s}^{\dot\gamma_s} \dot\beta_t
,\dot\gamma_s)+g_{\dot\gamma_s}(D_{\gamma_s}^{\dot\gamma_s} \dot\beta_t
,D_{\beta_t}^{\dot\gamma_s}\dot\gamma_s)\right)dt.
\end{multline*}
 Since $D_{\beta_t}^{\dot\gamma_s}D_{\gamma_s}^{\dot\gamma_s} \dot\beta_t=D_{\gamma_s}^{\dot\gamma_s}D_{\beta_t}^{\dot\gamma_s} \dot\beta_t-R^{\gamma_s}(\dot\gamma_s,\dot\beta_t)\dot\beta_t$ (see \eqref{jacobiop}),
\begin{multline}\label{almost}
\frac{d^2}{ds^2}E(\gamma_s)=\int_a^b g_{\dot\gamma_s}(D_{\gamma_s}^{\dot\gamma_s}D_{\beta_t}^{\dot\gamma_s} \dot\beta_t-R^{\gamma_s}(\dot\gamma_s,\dot\beta_t)\dot\beta_t
,\dot\gamma_s)dt\\+\int_a^b g_{\dot\gamma_s}(D_{\gamma_s}^{\dot\gamma_s} \dot\beta_t
,D_{\beta_t}^{\dot\gamma_s}\dot\gamma_s) dt.
\end{multline}
Finally,  as $\gamma_0=\gamma$ is a geodesic, we have 
$g_{\dot\gamma_s}(D_{\gamma_s}^{\dot\gamma_s}D_{\beta_t}^{\dot\gamma_s} \dot\beta_t,
\dot\gamma_s)=\frac{d}{dt}(g_{\dot\gamma_s}(D_{\beta_t}^{\dot\gamma_s} \dot\beta_t,
\dot\gamma_s))$ for $s=0$, and using this in \eqref{almost}, integrating
and recalling \eqref{commut}, we get \eqref{secondvariation}.
\end{proof}
Observe that the transverse acceleration $D^{\dot\gamma}_{\beta_t}\dot\beta_t|_{s=0}$ depends not only on the vector field $W$ along $\gamma$, but  also  on the variation $\Lambda$. We will see later that the dependence on the variation disappears when we put certain boundary conditions.
\subsection{Submanifolds and second fundamental form}\label{submani}
We refer the reader to \cite{oneill} for the basic notions and notation on submanifolds in semi-Riemannian manifolds.
Let us assume that $(M,L)$ is a pseudo-Finsler manifold and $P\subset M$ a submanifold of $M$. We  denote the tangent bundle of $P$ as $TP$ and define the {\it normal bundle} $TP^\perp$ of  $P$ as the  set of  vectors $v\in A$ such that  $\pi(v)\in P$ and $g_v(v,w)=0$ for every $w\in T_{\pi(v)}P$.  We  write $T_pP^\perp=TP^\perp\cap T_pM$ for every $p\in P$.  This  is a conic subset, namely, if $v\in T_pP^\perp$, then $\lambda v\in T_pP^\perp$ for every $\lambda>0$.   Notice that even though $TP^\perp$  is not necessarily a fiber bundle over $P$, it admits a structure of submersion.  By abuse of notation, we write also $\pi$ for the restriction to $TP^\perp$ of the natural projection $\pi:TM\rightarrow M$. Let $P_0:=\{p\in P: \exists v\in TP^\perp, \pi(v)=p\}$, $r:=\dim P$, and recall that $n=\dim M$.

\begin{lemma}
 If $TP^\perp$ is not empty, then it is an $n$-dimensional submanifold 
of $TM$. The subset $P_0$ is open in $P$, and the map 
$\pi:TP^\perp\rightarrow P_0$ is a submersion. In particular, for every $p\in P$,
$T_pP^\perp$ is a submanifold of $T_pM$ of dimension $n-r$. 
\end{lemma}
\begin{proof}
  Assume  that $E_1,\ldots,E_r$ is a local frame field over an open subset $\Xi$ of $P$  , and observe that $A\cap \pi(\Xi)^{-1} $ is a submanifold of $TM$ of dimension $n+r$. Define $\varphi:A\cap\pi(\Xi)^{-1} \rightarrow \R^r$ as $\varphi(v)=(g_v(v,E_1),\ldots,g_v(v,E_r))$. 
Observe that if $h:(-\epsilon,\epsilon)\rightarrow T_pM$ is a curve such that $h(0)=v$ and $\dot h(0)=u$ and $w\in T_pM$,  then using the covariant derivative along the constant curve equal to $p$,  \eqref{almostmetric} and Remark \ref{porvenir}, we get
\[\frac{d}{dt}g_{h(t)}(h(t),w)|_{t=0}= g_v(u,w)+ 2 C_v(u,v,w)=g_v(u,w)\]
and then  the fiber derivative of $\varphi$ is given by
  $D_f\varphi_v(u)=(g_v(u,E_1),\ldots,g_v(u,E_r))$ for every $u\in T_{\pi(v)}M$. As $g_v$ is a non-degenerate  scalar product and $(E_1,\ldots,E_r)$ is linearly independent, this ensures that the map $\varphi$ is a submersion. Then \linebreak $TP^\perp\cap \pi^{-1}(\Xi)=\varphi^{-1}(0)$, which, if not empty, is an $n$-dimensional submanifold of $TM$, and $TP^\perp\cap T_pM$ is a submanifold of $T_pM$ of dimension $n-r$. This implies that $P_0=\pi(TP^\perp)$  is open in $P$ and that 
$\pi:TP^\perp\rightarrow P_0$ is a submersion.
 \end{proof}
We denote  by $\mathcal{F}(P)$ the space of smooth real functions on $P$,  by  $\mathfrak{X}(P)$ the $\mathcal{F}(P)$-module  space of 
smooth sections of the fiber bundle $TP$ over $P$ and by $\mathfrak{X}(P)^\perp$
the space of  smooth sections of $\pi:TP^\perp\rightarrow P_0$. 
 Given $N\in \mathfrak{X}(P)^\perp$, we denote by $\mathfrak{X}(P)^\perp_N$ 
the subset of smooth sections $W$ of  $\pi:i^*(TM)\rightarrow P$ (where $i^*(TM)$  is the pull-back of $TM$ along the inclusion $i:P\rightarrow M$) 
such that, for every $p\in P$, $W(p)$ is $g_{N}$-orthogonal to $T_pP$. Observe that in particular $N\in \mathfrak{X}(P)^\perp_N$.  Then if $g_N|_{T_pP\times T_pP}$ is nondegenerate, we have the decomposition
\begin{equation}\label{decompos}
T_pM=T_pP\oplus (T_pP)^\perp_N,
\end{equation}
where $(T_pP)^\perp_N$ is the subspace of $T_pM$ consisting of $g_N$-orthogonal vectors to $T_pP$. Then for every smooth section $V$ of $\pi:TM\rightarrow P$, we can define ${\rm tan}_N(V)$ (resp. ${\rm nor}_N(V)$)  as the vector field in $\mathfrak{X}(P)$ obtained  by  projecting $V(p)$ to $T_pP$ (resp.  $(T_pP)^\perp_N$), for every $p\in P$,  through the decomposition \eqref{decompos}.
\begin{defi}\label{secondfundamental}
Fix $N\in \mathfrak{X}(P)^\perp$ and suppose that $g_N|_{T_pP\times T_pP}$ is nondegenerate for every $p\in P$. Then
\begin{enumerate}[(i)]
\item the {\it second fundamental form} of $P$ in the direction of $N$ 
 is the map  \linebreak 
$S^P_N:\mathfrak{X}(P)\times \mathfrak{X}(P)\rightarrow  
\mathfrak{X}(P)^\perp_N$ given by $S_N^P(U,W)={\rm nor}_N \nabla^N_UW$,
\item the {\it normal second fundamental form}  $\tilde{S}^P_N:\mathfrak{X}(P)\rightarrow  \mathfrak{X}(P)$  is given by \linebreak
$\tilde{S}_N^P(U)={\rm tan}_N\nabla^N_UN$.
\end{enumerate}
\end{defi}
\begin{prop}
With the above notation,  $S^P_N$ is ${\mathcal F}(P)$-bilinear and symmetric 
and $\tilde{S}_N^P$ is ${\mathcal F}(P)$-linear. Moreover,
\begin{equation}\label{secondfund}
g_N(S_N^P(U,W),N)=-g_N(\tilde{S}_N^P(U),W)
\end{equation}
for every $U,W\in \mathfrak{X}(P)$.
\end{prop}
\begin{proof}
 First we show that $S^P_N$ is ${\mathcal F}(P)$-bilinear. This is 
immediate for the first variable. As to the second one, 
let $f\in{\mathcal F}(P)$ and $U,W\in\mathfrak{X}(P)$. Then
\[
S_N^P(U,fW)={\rm nor}_N\nabla^N_U(fW)={\rm nor}_N(U(f)W+f\nabla^N_U(W))=f\,{\rm nor}_N(\nabla^N_UW),
\]
since $W$ is tangent to $P$. For the symmetry, 
\[S_N^P(U,W)-S_N^P(W,U)={\rm nor}_N(\nabla^N_UW-\nabla^N_WU)={\rm nor}_N[U,W]=0,\]
since $[U,W]$ is tangent to $P$. Again it is straightforward to check that 
$\tilde{S}^P_N$ is \linebreak ${\mathcal F}(P)$-linear. For \eqref{secondfund}, using that the Chern connection is almost $g$-compatible, $g_N(N,W)=0$ and Remark \ref{porvenir}, we get
\begin{multline*}
g_N(S_N^P(U,W),N)=g_N(\nabla^N_UW,N)\\=-g_N(W,\nabla^N_UN)-2C_N(\nabla^N_UN,W,N)=-g_N(\tilde{S}_N^P(U),W),
\end{multline*}
as required.
\end{proof}
\begin{rem}\label{homosecond}
Observe that from the homogeneity of the Chern connection it \linebreak follows 
that $S_{\lambda N}^P=S_N^P$, and then from \eqref{secondfund}, that $\tilde{S}_{\lambda N}^P=\lambda\tilde{S}_N^P$. Moreover, we shall interpret $S^P_N$ and $\tilde{S}_N^P$ as maps $S_N^P:T_pP\times T_pP\rightarrow (T_pP)_N^\perp$ and $\tilde{S}_N^P:T_pP\rightarrow T_pP$, respectively, even when $\mathfrak{X}(P)^\perp$ is empty (think that $P$ could be non-orientable), since if $T_pP^\perp$ is not empty, then there is some open subset $\Xi\subset P$ such that  $\mathfrak{X}(\Xi)^\perp$ is not empty.
\end{rem}
\subsection{The endmanifold case}\label{endmanifold}
Consider now the space of curves \[C_L(P,Q)\subset C_L(M,[a,b])\] joining two submanifolds $P$ and $Q$ of $M$, namely,
\[C_L(P,Q):=\{\gamma\in C_L(M,[a,b]): \gamma(a)\in P,\gamma(b)\in Q\}.\]
 When we consider a piecewise smooth $(P,Q)$-variation of $\gamma\in C_L(P,Q)$ 
by curves in $C_L(P,Q)$, the variational vector field is tangent to $P$ and 
$Q$ at the endpoints. Indeed, we define 
 \[T_\gamma C_L(P,Q)=\{W\in T_{\gamma}C_L(M,[a,b]): W(a)\in 
T_{\gamma(a)}P,W(b)\in T_{\gamma(b)}Q\}.\]
 Moreover, we say that $\gamma$ is a critical point of $E|_{C_L(P,Q)}$ if $dE_\gamma(W)=0$ for every $W\in T_\gamma C_L(P,Q)$.
\begin{cor}\label{critgeo}
Let $\gamma\in C_L(P,Q)$ and assume that the Legendre transformation $\mathscr{L}_L$ is injective. Then $\gamma$ is a critical point of the energy functional $E|_{C_L(P,Q)}$ if and only if $\gamma$ is a geodesic $g_{\dot\gamma}$-orthogonal to $P$ and $Q$.
\end{cor}
\begin{proof}
Let $t_0\in (a,b)$ an instant where $\gamma$ is smooth.  As the scalar product $g_{\dot\gamma}$ is nondegenerate, if we assume that $D_\gamma^{\dot\gamma}\dot\gamma\not=0$, using bump functions, we can find  a vector field $W$ such
that 
$g_{\dot\gamma}(D_\gamma^{\dot\gamma}\dot\gamma,W)>0$ in a neighborhood of $t_0$ that does not contain breaks and is zero everywhere. Then using \eqref{firstvariation}, we get a contradiction. Thus $\gamma$ must be a piecewise geodesic.
Assume that  $\dot\gamma(t_i^+)\not=\dot\gamma(t_i^-)$ for some $i=1,\ldots,h$. As $\mathscr{L}_L$ is assumed to be injective, we can  find a variational vector field $W_i$ such that 
$\mathscr{L}_L(\dot\gamma(t_i^+))(W_i)-\mathscr{L}_L(\dot\gamma(t_i^-))(W_i)\not=0$ and it is zero at the other breaks. This gives a contradiction in \eqref{firstvariation}, since $\gamma$ is a critical point. Therefore, $\gamma$ is a geodesic. Finally given $w\in T_{\gamma(a)}P$, construct a vector field $W$ such that $W(a)=w$ and $W(b)=0$. Then \eqref{firstvariation} implies that $g_{\dot\gamma(a)}(\dot\gamma(a),w)=0$. Analogously, we can show that for any $v\in T_{\gamma(b)}Q$, $g_{\dot\gamma(b)}(\dot\gamma(b),v)=0$.  The converse is trivial.
\end{proof}
\begin{cor}
Let $\gamma\in C_L(P,Q)$ be a geodesic of $(M,L)$ that is 
$g_{\dot\gamma}$-orthogonal to $P$ and $Q$ at the endpoints and such that 
$g_{\dot\gamma(a)}|_{P\times P}$ and $g_{\dot\gamma(b)}|_{Q\times Q}$ are nondegenerate. 
 Consider a smooth $L$-admissible $(P,Q)$-variation. Then 
\begin{multline*}
E''(0)=\int_a^b \left(-g_{\dot\gamma}(R^\gamma(\dot\gamma,W)W,\dot\gamma)+g_{\dot\gamma}(W',W')\right)dt\\+
g_{\dot\gamma(b)}(S^P_{\dot\gamma(b)}(W,W),\dot\gamma(b))
-g_{\dot\gamma(a)}(S^Q_{\dot\gamma(a)}(W,W),\dot\gamma(a)),
\end{multline*}
where $W$ is the variational vector field of the variation along $\gamma$.
\end{cor}
\begin{proof}
It is a straightforward consequence of the definition of the second fundamental form in Definition \ref{secondfundamental} and \eqref{secondvariation}.
\end{proof}
\subsection{The index form}
When $\gamma$ is a geodesic of a pseudo-Finsler manifold $(M,L)$  such that
it is $g_{\dot\gamma}$-orthogonal to $P$ and $Q$ at the endpoints and such that $g_{\dot\gamma(a)}|_{P\times P}$ and $g_{\dot\gamma(b)}|_{Q\times Q}$ are nondegenerate, we can define
  the $(P,Q)$-index form  of $\gamma$ as
\begin{multline*}
I^{\gamma}_{P,Q}(V,W)=\int_a^b \left(-g_{\dot\gamma}(R^\gamma(\dot\gamma,V)W,\dot\gamma)+g_{\dot\gamma}(V',W')\right)dt\\
+g_{\dot\gamma(b)}(S^P_{\dot\gamma(b)}(V,W),\dot\gamma(b))
-g_{\dot\gamma(a)}(S^Q_{\dot\gamma(a)}(V,W),\dot\gamma(a)),
\end{multline*}
where $V,W\in T_\gamma C_L(P,Q)$.
\begin{rem}
Observe that when $P$ (resp. $Q$) is a hypersurface of $M$ and $\gamma$ is orthogonal to $P$ (resp. to $Q$),  the  nondegeneracy condition on $g_{\dot\gamma(a)}|_{P\times P}$  (resp. $g_{\dot\gamma(b)}|_{Q\times Q}$) is equivalent to $L(\dot\gamma(a))\not=0$ (resp. $L(\dot\gamma(b))\not=0$).
\end{rem}
Let us now give a useful property of the tensor $B^V$ defined in subsection \ref{curvature}.
\begin{lemma}\label{Bzero}
 Let $(M,L)$ be a pseudo-Finsler manifold  
and $\Omega$ some open subset of $M$  and $V$ an $L$-admissible vector field 
in $\Omega$.  Then for every $X,Y,Z,W \in\mathfrak{X}(\Omega)$, 
$B^V$ satisfies
\begin{enumerate}[(i)]
\item If $Z=V$ or $W=V$, then
\[  B^V(X,Y,Z,W) = 0.  \]
\item If at least two of 
the vector fields $X,Y,Z,W$ are equal to $V$, then
\[ B^V(X,Y,Z,W) = 0. \]
\item For any curve $\gamma$ and $U,P\in \mathfrak{X}(\gamma)$, 
\[ g_{\dot\gamma}(R^\gamma(\dot\gamma,U)\dot\gamma,P) = 
-g_{\dot\gamma}(R^\gamma(\dot\gamma,U)P,\dot\gamma) \]
\end{enumerate}
\end{lemma}
\begin{proof}
Since $B^V$ is symmetric in the last two arguments, to prove $(i)$ 
it is enough to show that $B^V(X,Y,V,W) = 0$.

 Using the definition of $B^V$ and Remark \ref{porvenir}, we get
\[
B^V(X,Y,V,W)=\nabla^V_YC_{V}\left(\nabla^V_XV,V,W\right)
-\nabla^V_XC_{V}\left(\nabla_Y^VV,V,W\right).
\]
Then, $(i)$ follows from the symmetry of the tensors 
$\nabla^V_YC_{V}$ and $\nabla^V_XC_{V}$ together with relation 
\eqref{formulilla}, because 
\begin{multline*}
B^V(X,Y,V,W) = \nabla^V_YC_{V}\left(V,\nabla^V_XV,W\right)
-\nabla^V_XC_{V}\left(V,\nabla_Y^VV,W\right) = \\
-C_{V}\left(\nabla^V_YV,\nabla^V_XV,W\right)+ 
C_{V}\left(\nabla^V_XV,\nabla_Y^VV,W\right)= 0~.
\end{multline*}
Now, $(ii)$ is a simple consequence of $(i)$ and the antisymmetry of 
$B^V$ in the first two arguments.

 $(iii)$, on its turn, is simply a generalization of \cite[Eq. (14)]{J14}, 
stating that it is valid for any curve $\gamma$, not only for geodesics. 
The proof is exactly the same. 

\end{proof}

\begin{prop}\label{kernelindex}
The kernel of $I^{\gamma}_{P,Q}$ consists of the vector fields \linebreak 
$V\in T_\gamma C_L(P,Q)$ satisfying 
\[V''=R^\gamma(\dot\gamma,V)\dot\gamma \quad,\quad
{\rm tan}_{\dot\gamma} V'(a)=\tilde{S}^P_{\dot\gamma(a)}(V(a)) ~~\mbox{and}~~
{\rm tan}_{\dot\gamma} V'(b)=\tilde{S}^Q_{\dot\gamma(b)}(V(b))~.
\]
\end{prop}
\begin{proof}
  Using the fact established in the above lemma, that \linebreak
$g_{\dot\gamma}(R^\gamma(\dot\gamma,V)W,\dot\gamma)=
-g_{\dot\gamma}(R^\gamma(\dot\gamma,V)\dot\gamma,W)$, and taking into account 
that, since the Chern connection is almost $g$-compatible and $\gamma$ is a 
geodesic,
\begin{equation*}
g_{\dot\gamma}(V',W')=
\frac{d}{dt}(g_{\dot\gamma}(V',W))-g_{\dot\gamma}(V'',W),
\end{equation*}
we may conclude that the index form can also be expressed as
\begin{multline*}
I^{\gamma}_{P,Q}(V,W)=\int_a^b \left(g_{\dot\gamma}(R^\gamma(\dot\gamma,V)\dot\gamma,W)-g_{\dot\gamma}(V'',W)\right)ds
+\left[g_{\dot\gamma}(V',W)\right]_a^b\\+g_{\dot\gamma(b)}(S^P_{\dot\gamma(b)}(V,W),\dot\gamma(b))-g_{\dot\gamma(a)}(S^Q_{\dot\gamma(a)}(V,W),\dot\gamma(a)).
\end{multline*}

This means that $V\in T_\gamma C_L(P,Q)$ belongs to the kernel of the index 
form if and only if 
$V''-R^\gamma(\dot\gamma,V)\dot\gamma=0$ and 
\begin{align*}
g_{\dot\gamma(a)}(V'(a),W(a))+g_{\dot\gamma(a)}(S^P_{\dot\gamma(a)}(V,W),\dot\gamma(a))&=0,\\
g_{\dot\gamma(b)}(V'(b),W(b))+g_{\dot\gamma(b)}(S^Q_{\dot\gamma(b)}(V,W),\dot\gamma(b))&=0,
\end{align*}
at the endpoints. 
Using \eqref{secondfund}, we get that 
\begin{equation*}
g_{\dot\gamma(a)}(V'(a)-\tilde{S}^P_{\dot\gamma(a)}(V(a)),u)=0,\quad
g_{\dot\gamma(b)}(V'(b)-\tilde{S}^P_{\dot\gamma(b)}(V(b)),w)=0
\end{equation*}
for every $u\in T_{\gamma(a)}P$ and  $w\in T_{\gamma(b)}Q$.  Thus ${\rm tan}_{\dot\gamma} V'(a)=\tilde{S}^P_{\dot\gamma(a)}(V(a))$ and \linebreak
${\rm tan}_{\dot\gamma} V'(b)=\tilde{S}^Q_{\dot\gamma(b)}(V(b))$ as required.
\end{proof}
\subsection{Jacobi fields, Conjugate and focal points.}\label{jacobea}
\begin{defi}
Given a geodesic $\gamma:[a,b]\rightarrow M$ of $(M,L)$, 
a {\it Jacobi field} of $\gamma$ is a vector field $J$ along $\gamma$ 
satisfying
\[J''=
R^\gamma(\dot\gamma,J)\dot\gamma.\]
Moreover, given a submanifold $P$ such that $\gamma(a)\in P$ and $\dot\gamma(a)$ is $g_{\dot\gamma(a)}$-orthogonal to $P$, we say that a Jacobi field is {\it $P$-Jacobi} if 
$J(a)$ is tangent to $P$ and ${\rm tan}_{\dot\gamma} J'(a)=\tilde{S}^P_{\dot\gamma(a)}(J(a))$. An instant $t_0\in (a,b]$ is  called
\begin{enumerate}[(i)]
\item  {\it conjugate} if there exists a Jacobi field $J$ along $\gamma$ such that $J(a)=J(t_0)=0$,
\item  $P$-{\it focal} if there exists a $P$-Jacobi field $J$ such that $J(t_0)=0$.
\end{enumerate}
\end{defi}
Given a geodesic $\gamma:[a,b]\rightarrow M$ of a pseudo-Finsler manifold, we 
can construct a parallel orthonormal frame field along $\gamma$. Indeed, 
 fix an orthonormal basis $e_1,e_2,\ldots,e_n$ of 
$(T_{\gamma(a)}M,g_{\dot\gamma(a)})$, namely, a basis satisfying 
$g_{\dot\gamma(a)}(e_i,e_j)=\varepsilon_i \delta_{ij}$, where $\varepsilon_i^2=1$, 
$\delta_{ij}$ is the Kronecker's delta and $i,j=1,\ldots,n$. Then define the 
parallel vector fields $E_1,E_2,\ldots,E_n$ along $\gamma$ such that $E_i(a)=e_i$ for every $i=1,\ldots,n$. The fact that $\gamma$ is a geodesic implies 
that $g_{\dot\gamma}(E_i,E_j)=\varepsilon_i \delta_{ij}$, since
\[(g_{\dot\gamma}(E_i,E_j))'=g_{\dot\gamma}(E'_i,E_j)+g_{\dot\gamma}(E_i,E'_j)+
2 C_{\dot\gamma}(D_\gamma^{\dot\gamma}\dot\gamma,E_i,E_j)=0,\]
where we have used that the Chern connection is almost $g$-compatible, $\gamma$ is a geodesic and the frame field is parallel along $\gamma$. 

 We recall that a variation is called {\it geodesic} if its longitudinal curves are geodesics. As in the pseudo-Riemannian case, we have
\begin{prop}\label{JacobVari}
 In any pseudo-Finsler manifold, the variation vector field of a geodesic variation is a Jacobi field.
\end{prop}
We note that this is true also in the more general setting of spray manifolds, see \cite[Section 8.2.2]{SLK}. 
\begin{lemma}\label{geodesicvariations}
Let $\gamma:[a,b]\rightarrow M$ be a geodesic of $(M,L)$ with $\gamma(a)=p$. Then for any $v,w\in T_pM$, there exists a unique Jacobi field such that $J(a)=v$ and $ J'(a)=w$.
\end{lemma}
Using a parallel orthonormal frame field along $\gamma$, the proof is the 
same as in the pseudo-Riemannian case; see, e.g., \cite[page 217]{oneill}.
\begin{prop}
Let $p$ be a point of a pseudo-Finsler manifold $(M,L)$. If 
$v\in T_pM\cap A$ belongs to the domain of $\exp_p$, then for any 
$w\in T_v(T_pM)$ we have
\[d\exp_p(v)[w]=J(1),\]
where $J$ is the unique Jacobi field on $\gamma_v$ such that $J(0)=0$ and $J'(0)=w$.
\end{prop}
\begin{proof}
Consider the variation $\tilde{\Lambda}(t,s)=t(v+sw)$ for $0\leq t\leq 1$ and $s$ small enough, and define 
\[\Lambda(t,s)=\exp_p(\tilde{\Lambda}(t,s))=\gamma_{v+sw}(t).\]
By 
Proposition \ref{JacobVari}, the variation vector field $J(t)=\Lambda_s(t,0)$
($t\in[0,1]$) is a Jacobi field. Moreover, as the curve $s\rightarrow \Lambda(0,s)=p$ is constant, $J(0)=0$ and if we denote by $\beta$ the constant curve with value $p$ and  write $\gamma_s=\gamma_{v+sw}$, from \eqref{commut} we get
\[J'(0)=
D_{\beta}^{\dot\gamma_s}\dot\gamma_s(0)=w,\]
since $\dot\gamma_s(0)=v+sw$.
\end{proof}
Given a geodesic $\gamma:[a,b]\rightarrow M$ of a pseudo-Finsler manifold, 
let \linebreak$\dot\gamma^\perp:=\{v\in T_{\dot\gamma}M: g_{\dot\gamma}(\dot\gamma,v)=0\}$. When $L\circ \dot\gamma\not=0$,  we have the decomposition  
\begin{equation}\label{decom}
T_{\dot\gamma}M={\rm span}(\dot\gamma)\oplus \dot\gamma^\perp.\end{equation}
 Moreover, if $Y\in {\mathfrak X}(\gamma)$, let us denote by ${\rm tan}_{\gamma}(Y)$ and ${\rm nor}_{\gamma}(Y)$ the first and second projection in the decomposition \eqref{decom}, respectively.
 \begin{lemma}\label{derivcommut}
 With the above notation, if $L\circ \dot\gamma\not=0$, then 
$({\rm tan}_{\gamma}(Y))'={\rm tan}_{\gamma}(Y')$ and 
$({\rm nor}_{\gamma}(Y))'={\rm nor}_{\gamma}(Y')$.
 \end{lemma}
 \begin{proof}
 It is enough to prove that $({\rm tan}_{\gamma}(Y))'={\rm tan}_{\gamma}(Y')$, since the other equality comes then from $Y={\rm tan}_{\gamma}(Y)+{\rm nor}_{\gamma}(Y)$. Observe that $g_{\dot\gamma}(Y,\dot\gamma)=g_{\dot\gamma}({\rm tan}_{\gamma}(Y),\dot\gamma)$. Using that $\gamma$ is a geodesic and the Chern connection is almost $g$-compatible, we get that
 $g_{\dot\gamma}(Y',\dot\gamma)=g_{\dot\gamma}(({\rm tan}_{\gamma}(Y))',\dot\gamma)$.
This concludes the proof because $g_{\dot\gamma}(\dot\gamma,\dot\gamma)=L(\dot\gamma)\not=0$.
 \end{proof}
\begin{lemma}\label{orthogJacobi}
Consider a vector field $J$ along a geodesic $\gamma:[a,b]\rightarrow M$. If $J$ is a Jacobi field, then:
\begin{enumerate}[(i)]
\item $J$ is tangent to $\gamma$ if and only if $J(t)=(a_1 t+a_2)\dot\gamma (t)$ with $a_1,a_2\in\R$;
\item the following statements are equivalent:
\begin{enumerate}
\item $g_{\dot\gamma}(\dot\gamma,J)=0$,
\item there exist real numbers $a,b$ such that 
\[g_{\dot\gamma(a)}(\dot\gamma(a),J(a))=g_{\dot\gamma(b)}(\dot\gamma(b),J(b))=0,\]
\item there exists a real number $a$ such that 
\[g_{\dot\gamma(a)}(\dot\gamma(a),J(a))=g_{\dot\gamma(a)}(\dot\gamma(a), J'(a))=0.\]
\end{enumerate}
\end{enumerate}
Moreover, if $\gamma$ is nonnull, that is, $L(\dot\gamma)\not=0$, then $J$ is a Jacobi field if and only if ${\rm nor}_{\gamma}(J)$ and ${\rm tan}_{\gamma}(J)$ are Jacobi fields.
\end{lemma}
\begin{proof}
For $ (i)$, observe that 
\begin{equation}\label{Rgamma}
R^\gamma(\dot\gamma,\dot\gamma)\dot\gamma=0
\end{equation}
 (see \cite[Theorem 2.2]{J14} and take into account that  $R^V$ is 
antisymmetric in the first two arguments). Then, for $J=f\dot\gamma$, the Jacobi equation reduces to $\frac{d^2f}{dt^2}=0$. To see $(ii)$, observe that
since $\gamma$ is a  geodesic,
\begin{equation}\label{RgammaJ}
(g_{\dot\gamma}(J,\dot\gamma))''=
g_{\dot\gamma}(J'',\dot\gamma)=g_{\dot\gamma}(R^\gamma(\dot\gamma,J)\dot\gamma,\dot\gamma)=0,
\end{equation}
 by part $(iii)$ of Lemma \ref{Bzero}.  Hence $g_{\dot\gamma}(J,\dot\gamma)=C_1t+C_2$, where $C_1$ and $C_2$ are real constants, and $g_{\dot\gamma}(J',\dot\gamma)=C_1$, thus $(ii)$ follows. For the last statement, observe that 
$R^\gamma(\dot\gamma,{\rm tan}_{\gamma}J)\dot\gamma=0$ and then 
$R^\gamma(\dot\gamma,J)\dot\gamma=R^\gamma(\dot\gamma,{\rm nor}_{\gamma}J)\dot\gamma$ (recall  that $R^\gamma(\dot\gamma,W)\dot\gamma$ is tensorial in $W$).  Using again that $g_{\dot\gamma}(R^\gamma(\dot\gamma,J)\dot\gamma,\dot\gamma)=0$ and Lemma \ref{derivcommut}, the Jacobi equation splits into the two equations 
\begin{align*}
({\rm tan}_{\gamma}J)''&=0&\text{and}&& ({\rm nor}_{\gamma}J)''&=R^\gamma(\dot\gamma,{\rm nor}_{\gamma}J)\dot\gamma. 
\end{align*}
Finally, applying part $(i)$, we conclude  the proof.
\end{proof}
\begin{prop}\label{symplectic}
If $J_1$ and $J_2$ are Jacobi fields along a geodesic $\gamma$, then
 the function  $g_{\dot\gamma}(J_1,J'_2)-g_{\dot\gamma}(J'_1,J_2)$ is constant.
\end{prop}
\begin{proof}
Using that $\gamma$ is a geodesic, that the Chern connection is almost $g$-compatible and that both $J_1$ and $J_2$ satisfy the Jacobi equation, we obtain
\begin{align}\label{eqJacconst}
( g_{\dot\gamma}(J_1,J'_2)-g_{\dot\gamma}(J'_1,J_2))'=&
g_{\dot\gamma}(J'_1,J'_2)+
g_{\dot\gamma}(J_1,J''_2)
\nonumber\\\,\,&-g_{\dot\gamma}(J'_1,J'_2)
-g_{\dot\gamma}(J''_1,J_2)\nonumber\\
=&g_{\dot\gamma}(J_1,R^\gamma(\dot\gamma,J_2)\dot\gamma)-g_{\dot\gamma}(R^\gamma(\dot\gamma,J_1)\dot\gamma,J_2).
\end{align}
 Let us see that the last quantity is zero. If one of the vector fields $J_1$, $J_2$ is proportional to $\dot\gamma$  in some interval,  then it follows from \eqref{Rgamma} and \eqref{RgammaJ}. Otherwise, if both of them are linearly independent to $\dot\gamma$ you can choose extensions $V$, $U_1$ and $U_2$ of $\dot\gamma$, $J_1$ and $J_2$, respectively, such that  $[U_1,V]=[U_2,V]=0$  and
\begin{multline*}
g_{\dot\gamma}(J_1,R^\gamma(\dot\gamma,J_2)\dot\gamma)-g_{\dot\gamma}(R^\gamma(\dot\gamma,J_1)\dot\gamma,J_2)=\\
g_{V}(U_1,R^V(V,U_2)V)-g_{V}(R^V(V,U_1)V,U_2)
\end{multline*}
(see \cite[Theorem 2.2]{J14}), but the last quantity is zero along $\gamma$ because of  \eqref{seisB} and  part $(ii)$ of 	Lemma \ref{Bzero}.  By continuity, we conclude that \eqref{eqJacconst} is zero everywhere. 
\end{proof}
\subsection{Remarks about Morse theory}
Let us observe that in principle, there should not be further obstructions to prove that geodesics of a pseudo-Finsler metric are critical points of the energy functional in a suitable infinite dimensional $H^1$-Sobolev space, for example, by generalizing the proof in \cite[Proposition 2.1]{CaJaMa11}. But in order to make Morse theory available, we need to overcome several problems. The first one is that  the Palais--Smale  condition only holds in general when the pseudo-Finsler metric is in fact a Finsler metric (with positive definite fundamental tensor), since it is well-known that Palais-Smale condition fails for semi-Riemannian metrics. The second problem is the differentiability of the energy function in the $H^1$ Sobolev space, because it is $C^2$ only when the pseudo-Finsler metric is semi-Riemannian (see \cite[Proposition 3.2]{AbbSch09} and \cite{Cap10}). This has been overcome in the case of Finsler metrics using that the energy functional is $C^2$ in the $C^1$-topology (see \cite{CJM10b,CJM10c}). The third problem is that when $A$ is strictly contained in $TM\setminus \bf 0$, the space of $L$-admissible curves can be non-complete, thus it seems interesting to study conditions of completeness in the pseudo-Finsler metric to guarantee the validity of the Morse theory as in \cite{CJM10b,CJM10c}. In \cite{CJS13} some results of geodesic connectedness of conic Finsler metrics are deduced using causality of spacetimes endowed with a Killing vector.  In the general case, when the fundamental tensor is allowed to have any signature, Lemma \ref{symplectic} is the key point to develop a relation between the spectral flow of a certain path of operators and the Maslov index of conjugate points as in \cite{PPT04}. In the presence of a Killing vector field, more precise results have been obtained in the Lorentzian realm \cite{CFS08,GP99,JMP10} and it is expectable to get similar results for Lorentzian Finsler metrics.

\end{document}